\documentclass{article}
\usepackage{amsmath,amssymb,amsthm}
\newtheorem{thm}{Theorem}[section] 
 
\newtheorem{defi}[thm]{Definition}

\newtheorem{lem}[thm]{Lemma}

\begin{document}

\title{\LARGE The exact order of the number of lattice points visible from the origin}
\author{\Large Wataru Takeda}
\date{\normalsize Department of Mathematics,\\ Kyoto University, \\Kitashirakawa Oiwake-cho, Sakyo-ku, Kyoto 606-8502,
Japan}
\maketitle
\begin{abstract}
We say a lattice point $X=(x_1,\ldots,x_m)$ is visible from the origin, if $\gcd(x_1,...,x_m)=1$. In other word, there are no other lattice point on the line segment from the origin $O$ to $X$. From J.E. Nymann's result \cite{Ny72}, we know that the number of lattice point from the origin in $[-r,r]^m$ is $(2r)^m/\zeta(m)+$(Error term). We showed that the exact order of the error term is $r^{m-1}$ for $m\ge3$. 
\end{abstract}

\section{Introduction}
The counting and probability problem of the visible lattice points is well known.

Let $\mathbb{V}^m=\{x=(x_1,\ldots,x_m)\in\mathbb{Z}^m~|~\text{$x$ is visible from the origin}\}$.
It is well known that the cardinality of the set $\mathbb{V}^m\cap\{(x_1,...,x_m)\in\mathbb{Z}^m~|~|x_i|\le r\ (1\le i\le m)\}$ is

\[\frac{2^m}{\zeta(m)}r^m+\left\{
\begin{array}{ll} 
O(r\log r)&(m=2)\\
O(r^{m-1})& (m\ge3),
\end{array}
\right.
\]
where $\zeta$ is the Riemann zeta function.
F. Mertens proved the the case of $m=2$ in 1874 \cite{Me74} and J. E. Nymann showed the case of $m\ge3$ in 1972 \cite{Ny72}.

In this article,
as a generalization of the result of Nymann, we study the number of elements of \[\mathbb{V}^m\cap\{(x_1,...,x_m)\in\mathbb{Z}^m~|~|x_i|\le r\ (1\le i\le m)\}.\]
Let $V_m(r)$ denote $|\mathbb{V}^m\cap\{(x_1,...,x_m)\in\mathbb{Z}^m~|~|x_i|\le r\ (1\le i\le m)\}|$ and let $E_m(x)$ denote the error term, i.e. \[E_m(r)=V_m(r)-\frac{2^m}{\zeta(m)}r^m.\] Then we obtain a generating function of $V_m(r)$ and the exact order of $E_m(r)$, where $m\ge3$.
More precisely, we prove 
\begin{thm}
If $m\ge3$, \[E_m(r)=\Omega(r^{m-1}).\]
\end{thm}
Combine Nymann's result \cite{Ny72} with this theorem, the exact order of the magnitude of $E_m(r)$ is $r^{m-1}$ for all $m\ge3$.
\section{Jordan totient function $\varphi(n)$} 

In \cite{Ap76}, it is proved that \[V_2(r)=8\sum_{n\le r}\varphi(n).\]

We follow this way and use the Jordan totient function $J_m(n)$ to obtain the value of $V_m(r)$.
\begin{defi} For $m\ge 1$ we define
\begin{align*}
J_m(n):=&|\{(x_1,...,x_m)\in\mathbb{Z}^m~|~\gcd(x_1,...,x_m,n)=1,1\le x_i\le n\ (1\le i\le m)\}|.\\
\intertext{For $m=0$ we define}
J_0(n):=&\left\{
\begin{array}{rl}
1&(n=1)\\
0&(n\neq1)
\end{array}
\right.
\end{align*}
\end{defi}

Since $J_1(n)=|\{x_1\in\mathbb{Z}~|~\gcd(x_1,n)=1,1\le x_1\le n\}|=\varphi (n)$, we regard $J_m(n)$ as a generalization of $\varphi(n)$.

The Euler totient function $\varphi(n)$ satisfies \[\varphi(n)=\sum_{d|n} \mu(d)\frac nd=n\prod_{\substack{p|n\\p:\text{prime}}}\left(1-\frac1p\right),\] where $\mu(d)$ is M\"{o}bius function. As well as this, $J_m(n)$ satisfies following lemma.

\begin{lem}
\label{thm:2.2}
$\displaystyle{J_m(n)=\sum_{d|n} \mu(d)\left(\frac nd\right)^m=n^m\prod_{\substack{p|n\\p:\text{prime}}}\left(1-\frac1{p^m}\right).}$
\end{lem}

\begin{proof}
It suffices to show that $\displaystyle{n^m=\sum_{d|n}J_m(d)}$, because the M\"{o}bius inversion formula gives the assertion.
Let $S$ denote $\{(x_1,...,x_m)\in \mathbb{Z}^m~|~1\le x_i\le n\ (1\le i\le m)\}$.
Let $S(d)$ be $\{(x_1,...,x_m)\in \mathbb{Z}^m~|~\gcd(x_1,...,x_m,n)=d,1\le x_i\le n\ (1\le i\le m)\}$, where $d$ divides $n$.
Then $S$ is the disjoint union $\displaystyle{S=\bigcup_{d|n}S(d)}$.

$\gcd(x_1,...,x_m,n)=d$ if and only if $\displaystyle{\gcd\left(\frac {x_1}d,...,\frac {x_m}d,\frac nd\right)=1}$, so $\displaystyle{|S(d)|=J_m\left(\frac nd\right)}$.
Hence we can write \[n^m=|S|=\sum_{d|n}|S(d)|=\sum_{d|n}J_m\left(\frac nd\right)=\sum_{d|n}J_m(d).\]
The equality on the left in Lemma follows.

Let us prove the other equality. If $n=1$ the product is empty and assigned to be the value $1$.

And if $n\ge 2$, we can observe that \[\sum_{d|n} \frac {\mu(d)}{d^m}=\prod_{\substack{p|n\\p:\text{prime}}}\left(1-\frac1{p^m}\right).\]
Thus the other equality in Lemma also follows.
\end{proof}

\section{Generating function of $V_m(r)$} 

\begin{thm}
\label{thm:main}
Generating function of $V_m(r)$ is the following.
\begin{align*}
\sum_{m=0}^{\infty}\frac{u^m}{m!}V_m(r)&=\frac1{2u}(e^{(2X+1)u}-e^{(2X-1)u})\\
\intertext{and}
\sum_{m=0}^{\infty}u^{m+1}V_m(r)&=\frac12\log \frac{1-(2X-1)u}{1-(2X+1)u}, \\
\intertext{where $\displaystyle{i\sum_{n \le r}J_{i-1}(n)}$ is replaced by $X^i$ when $i\ge1$, and $X^0$ are assigned to be the value $0$.}
\end{align*}
\end{thm}

\begin{proof}

It suffices to show that \[V_m(r)=\frac1{2(m+1)}\{(2X+1)^{m+1}-(2X-1)^{m+1}\}.\]
Let $V_m^+(r)$ denote $|\mathbb{V}^m\cap\{(x_1,...,x_m)~|~0< x_i\le r\ (1\le i\le m)\}|$ for $m\ge1$.
Considering the sign of component, we have 
\[V_m(r)=\sum_{i=0}^{m-1}\binom mi 2^{m-i}V_{m-i}^+(r).\]

Let $A_m^+(n)=V_m^+(n)-V_m^+(n-1)$, then\[V_m^+(r)=\sum_{2 \le n \le r}A_m^+(n)+1.\]
We compute $A_m^+(n)$ in a combinatorial way as follows.

Fix $i$ with $0\le i<m$. Fix $I\subset\{1,\ldots,m\}$ such that $|I|=m-i$, and let
\[
 V=\{(x_1,\ldots,x_m)\in\mathbb{V}^m~|~x_j=n\text{ for all }j\in I, 0<x_j<n\text{ for all }j\not\in I\}.
\]
Then $|V|=J_{i}(n)$.
There are $\displaystyle{\binom m{m-i}}$ ways to choose $I\subset\{1,\ldots,m\}$ with $|I|=m-i$.
So the number of points $(x_1,\ldots,x_m)\in \mathbb{V}$ such that $|\{i~|~x_i=n \}|\ge m-i$ is \[\binom m{m-i} J_{i}(n).\]
But we count same point $(x_1,\ldots,x_m)\in \mathbb{V}$ such that $|\{i~|~x_i=n \}|=k$, $\displaystyle{\binom k{m-i}}$ times each $i\ (m-k\le i\le m-1)$.

We can show that $\displaystyle{(-1)^{m-1}\sum_{i=m-k}^{m-1}(-1)^{i}\binom k{m-i}=1}$, so we can count all points in $\mathbb{V}$ without repetition and obtain that $\displaystyle{A_m^+(n)=\sum_{j=0}^{m-1}(-1)^{m-1-j}\binom {m}jJ_j(n)}.$

\begin{align*}
\intertext{Thus we get}
V_m^+(r)&=\sum_{2\le n \le r}\sum_{j=0}^{m-i-1}(-1)^{m-i-1-j}\binom {m-i}jJ_j(n)+1,\\
\intertext{and}
V_m(r)&=\sum_{i=0}^{m-1}\binom mi 2^{m-i}\left(\sum_{2\le n \le r}\sum_{j=0}^{m-i-1}(-1)^{m-i-1-j}\binom {m-i}jJ_j(n)+1\right). \\
\intertext{For $j\ge 1$, $X^j$ is defined by $\displaystyle{j\sum_{n \le r}J_{j-1}(n)}$. To use this notation, we add the term $n=1$ to above sum. By the definition of $J_j(n)$, we get $J_j(1)=1$ for all $j$. And we  get $\displaystyle{\sum_{j=0}^{m-i-1}(-1)^{m-i-1-j}\binom {m-i}j=1}$ by applying the binomial theorem. Using this we have}
V_m(r)&=\sum_{i=0}^{m-1}\binom mi 2^{m-i}\left(\sum_{n \le r}\sum_{j=0}^{m-i-1}(-1)^{m-i-1-j}\binom {m-i}jJ_j(n)\right).\\
&=\sum_{i=0}^{m-1}\binom mi 2^{m-i}\left(\sum_{j=0}^{m-i-1}(-1)^{m-i-1-j}\binom {m-i}j\sum_{n \le r}J_j(n)\right).\\
\intertext{By the definition of $X^j$, we find $\displaystyle{\binom {m-i}j \sum_{n \le r}J_j(n)=\frac1{m-i+1}\binom {m-i+1}{j+1}\frac{X^{j+1}}{j+1}}$, this gives}
&=\sum_{i=0}^{m-1}\binom mi 2^{m-i}\frac 1{m-i+1}\sum_{j=0}^{m-i-1}(-1)^{m-i-1-j}\binom {m-i+1}{j+1}X^{j+1}.\\
\intertext{Replacing the index $j+1$ by $j$, we obtain}
&=\sum_{i=0}^{m-1}\binom mi 2^{m-i}\frac 1{m-i+1}\sum_{j=1}^{m-i}(-1)^{m-i-j}\binom {m-i+1}j X^j.
\end{align*}
Because the term $i=m$ and $m+1$ and $j=0$ of above sum are a constant times $X^0$, we get
\begin{align*}
V_m(r)&=\sum_{i=0}^{m+1}\binom mi 2^{m-i}\frac 1{m-i+1}\left(X^{m-i+1}-\sum_{j=0}^{m-i+1}(-1)^{m-i+1-j}\binom {m-i+1}j X^j\right).\\
\intertext{Applying the binomial theorem by repetition, we get}
&=\sum_{i=0}^{m+1}\binom mi 2^{m-i}\frac 1{m-i+1}\left\{X^{m-i+1}-(X-1)^{m-i+1}\right\},\\
&=\frac 1{2(m+1)}\sum_{i=0}^{m+1}\binom {m+1}i\left\{(2X)^{m+1}+(2X-2)^{m+1}\right\},\\
&=\frac 1{2(m+1)}\left\{(2X+1)^{m+1}+(2X-1)^{m+1}\right\}.
\end{align*}
This proves the theorem.
\end{proof}
\section{The value of $\displaystyle{\sum_{n \le r}J_{i-1}(n)}$} 
In this paper, we use the $\Omega$ simbol introduced by G.H. Hardy and J.E. Littlewood. This simbol is defined as follows:
\[f(x)=\Omega(g(x))\Longleftrightarrow\limsup_{x\rightarrow\infty}\left|\frac{f(x)}{g(x)}\right|>0\]
If there exists a function $g(x)$ such that $f(x)=O(g(x))$ and $f(x)=\Omega(g(x))$ then the exact order of $f(x)$ is $g(x)$. 
\begin{lem}
\label{thm:4.1}
Let $\{x\}$ be the fractional part of $x$.
If $m \ge 2$, \[\sum_{d\le r}\mu(d)\left(\frac rd\right)^m\left\{\frac rd\right\}=\Omega(r^m).\]
\end{lem}

\begin{proof}
It suffices to show that $\displaystyle{\sum_{d\le r}\frac {\mu(d)}{d^m}\left\{\frac rd\right\} \le M<0}$ for infinity many values of $r$ and some negative $M$.
\begin{align*}
\intertext{If $m \ge 4$ and $r$ is odd integer and greater than or equal to $3$,}
\sum_{d\le r}\frac {\mu(d)}{d^m}\left\{\frac rd\right\}&=-\frac1{2^{m+1}}+\sum_{3\le d\le r}\frac {\mu(d)}{d^m}\left\{\frac rd\right\}.\\
\intertext{Since $\mu(d)=1.0.-1$ and $\displaystyle{\left\{\frac rd\right\}\le 1}$,}
&\le -\frac1{2^{m+1}}+\sum_{3\le d\le r}\frac 1{d^m},\\
&=-\frac1{2^{m+1}}+\zeta(m)-1-\frac 1{2^m}<0,\\
\intertext{since when $m \ge 4$, $\displaystyle{\zeta(m)-1-\frac1{2^m}<\frac1{2^{m+1}}}$.}
\intertext{So for $m \ge 4$ the lemma follows.}
\intertext{Suppose that $m=2$ or $m=3$ and $\displaystyle{r=k\prod_{p\le 100}p}$, where the product is extended over all odd primes less than $100$ and $k$ isn't a multiple of $2$ and $p$.}
\intertext{Then,}
\sum_{d\le r}\frac {\mu(d)}{d^m}\left\{\frac rd\right\}&\le \sum_{d=1}^{100}\frac {\mu(d)}{d^m}\left\{\frac rd\right\}+\sum_{d=101}^{\infty}\frac {\mu(d)}{d^m}\left\{\frac rd\right\}.\\
\intertext{Since $\mu(d)=1.0.-1$ and $\displaystyle{\left\{\frac rd\right\}\le 1}$,}
&\le -\frac1{2^{m+1}}+\sum_{d=3}^{100}\frac {\mu(d)}{d^m}\left\{\frac rd\right\}+\sum_{d=101}^{\infty}\frac 1{d^m}.
\end{align*}
Because we defined $\displaystyle{r=k\prod_{p\le 100}p}$,  we obtain
\begin{align*}
\sum_{d=3}^{100}\frac {\mu(d)}{d^m}\left\{\frac rd\right\}&=\sum_{p=prime}^{47}\frac 1{(2p)^m}\frac 12-\frac12\left(\frac1{30^m}+\frac1{42^m}+\frac1{66^m}+\frac1{78^m}+\frac1{70^m}\right)\\
&<\frac25\times\frac1{10^{m-1}}
\end{align*}
From this result and $\displaystyle{\sum_{d=101}^{\infty}\frac 1{d^m}\le \frac1{100^{m-1}}}$, we find
\begin{align*}
\sum_{d\le r}\mu(d)\frac 1{d^m}\left\{\frac rd\right\}&\le -\frac1{2^{m+1}}+\frac25\times\frac1{10^{m-1}}+\frac1{100^{m-1}}\\
&<-\frac1{20},
\end{align*}
so for $m=2$ or $m=3$ the lemma follows.
This completes the proof of the lemma.
\end{proof}

\begin{lem}
\label{thm:4.2}
If $i \ge 3$,\ \ $\displaystyle{X^i=\frac1{\zeta(i)}r^i+\Omega(r^{i-1})}$
\end{lem}

\begin{proof}
For $i\ge 3$, $X^i$ is defined by $\displaystyle{i\sum_{n \le r}J_{i-1}(n)}$.
\begin{align*}
\intertext{From Lemma \ref{thm:2.2} we know that $\displaystyle{J_{i-1}(n)=\sum_{d|n} \mu(d)\left(\frac nd\right)^{i-1}}$,}
X^i&=i\sum_{n \le r}\sum_{d|n} \mu(d)\left(\frac nd\right)^{i-1}.\\
\intertext{We write $n=dq$ and sum over all pair of positive integers $d,q$ with $dq\le r$, thus}
X^i&=i\sum_{dq \le r} \mu(d)q^{i-1}\\
\intertext{Changing the order of summation,}
X^i&=i\sum_{d \le r}\mu(d)\sum_{q\le \frac rd}q^{i-1}\\
\intertext{Applying the relationship between Bernoulli numbers $B_0,B_1(=\displaystyle{\frac12}),B_2,\ldots$ and a sum $1^k+2^k+\cdots+n^k$, that is, $\displaystyle{\sum_{q=1}^n q^{i-1} =\frac1i\sum_{j=0}^{i-1}\binom ij B_j n^{i-j}}$,}
X^i&=i\sum_{d \le r}\mu(d)\frac 1i \sum_{j=0}^{i-1} \binom ij B_j\left[\frac rd\right]^{i-j},\\
\intertext{where $[x]$ is the greatest integer less than or equal to $x$. Now we use a relation $[x]=x-\{x\}$,}
X^i&=\sum_{d \le r}\mu(d)\sum_{j=0}^{i-1} \binom ij B_j\left(\frac rd-\left\{\frac rd\right\}\right)^{i-j}.
\end{align*}
We note that
\begin{align*} 
\left|\sum_{d\le r}\frac {\mu(d)}{d^m}\left\{\frac rd\right\}^k\right|&\le \sum_{d\le r}\frac 1{d^m}=\left\{
\begin{array}{cl}
\zeta(m)+O(r^{1-m})& (m\ge2)\\
\log(r)+\gamma+o(1) & (m=1),
\end{array}
\right.
\end{align*}
where $\gamma$ is Euler's constant, defined by the equation \[\gamma=\lim_{n\rightarrow\infty}\left(\sum_{k=1}^n\frac1k-\log n\right).\]

\noindent
Combining this result with Lemma \ref{thm:4.1} and $\displaystyle{\sum_{d\le r}\frac{\mu(d)}{d^i}=\frac1{\zeta(i)}+O\left(\frac1{r^{i-1}}\right)}$, where $i>1$, we get
\begin{align*}
X^i&=r^i\sum_{d \le r}\frac {\mu(d)}{d^i}+\Omega(r^{i-1}).\\
&=\frac1{\zeta(i)}r^i+\Omega(r^{i-1})
\intertext{This proved the lemma.}
\end{align*}
\end{proof}

\section{The exact order of magnitude of $E_m(r)$}
By using lemmas in last section with Theorem \ref{thm:main}, we prove following theorem about the exact order of magnitude of $E_m(r)$.
\begin{thm}
If $m\ge3$, \[E_m(r)=\Omega(r^{m-1}).\]
\end{thm}

\begin{proof}

\begin{align*}
\intertext{From Theorem \ref{thm:main}, }
V_m(r)&=\frac1{2(m+1)}\{(2X+1)^{m+1}-(2X-1)^{m+1}\},\\
&=(2X)^{m}+O(X^{m-2}).\\
\intertext{Applying Lemma \ref{thm:4.1} and \ref{thm:4.2}, we find}
V_m(r)&=\frac{2^m}{\zeta(m)}r^m+\Omega(r^{m-1}).
\end{align*}
\end{proof}
Combine Nymann's result \cite{Ny72} with this theorem, the exact order of the magnitude of $E_m(r)$ is $r^{m-1}$ for all $m\ge3$.

\end{document}